\documentclass[12pt]{amsart} 
\usepackage{amsmath,amstext,amsfonts,amssymb% 
} 

\newtheorem{theorem}{\bf Theorem}[section] 
 
\newtheorem{conj}[theorem]{\bf Conjecture} 
\newtheorem{corol}[theorem]{\bf Corollary} 
\newtheorem{defi}[theorem]{\bf Definition} 
 
\newtheorem{lemma}[theorem]{\bf Lemma} 
\newtheorem{notation}[theorem]{\bf Notation} 
\newtheorem{prop}[theorem]{\bf Proposition} 
 
\newtheorem{remark}[theorem]{\bf Remark}

\newcommand{\abs}[1]{\lvert#1\rvert} 
 
\newcommand{\bl}{{\noi$\bullet$\ }} 
 
\newcommand{\ctl}{\centerline} 
\newcommand{\noi}{{\noindent}}

\newcommand{\ifff}{{if and only if }} 
\newcommand{\la}{{\langle}} \newcommand{\ra}{{\rangle}} 
\newcommand{\lf}{{\lfloor}} \newcommand{\rf}{{\rfloor}} 
 
\newcommand{\nd}{{\text{ and }}} 

\DeclareMathOperator{\sgn}{sgn} 
\newcommand{\sm}{{\smallsetminus}} 
\newcommand{\stx}{\begin{smallmatrix}} \newcommand{\estx}{\end{smallmatrix}} 
\DeclareMathOperator{\Supp}{Supp} 
 
\newcommand{\thh}{th} 
  
\newcommand{\wdt}{\widetilde} 
\newcommand{\tld}{{\hbox{\!\raise-.4ex\hbox{\,$\,\wdt{}\,$\,}}}} 
\DeclareMathOperator{\wt}{wt}

\renewcommand{\a}{{\alpha}} 
\newcommand{\g}{{\gamma}} 
\newcommand{\Lb}{{\Lambda}} 
\newcommand{\lb}{{\lambda}} 

\newcommand{\vf}{{\varphi}}

\newcommand{\A}{{\mathbb A}} 
\newcommand{\D}{{\mathbb D}} 
\newcommand{\E}{{\mathbb E}} 
\newcommand{\F}{{\mathbb F}} 
\newcommand{\Q}{{\mathbb Q}} 
\newcommand{\R}{{\mathbb R}} 
\newcommand{\Z}{{\mathbb Z}}

\newcommand{\cB}{{\mathcal B}}

\newcommand{\cG}{{\mathcal G}}

\newcommand{\cS}{{\mathcal S}}

\DeclareMathOperator{\Gram}{Gram}

\begin{document} 

\title{Hermite versus Minkowski} 
\author[J. Martinet]{Jacques Martinet ($^*$)} 
\address{% 
Universit\'e de Bordeaux, 
Institut de Math\'ematiques, 
351, cours 
\linebreak\indent 
de la Lib\'e\-ration, 
33405 Talence cedex, France} 
\email{Jacques.Martinet@math.u-bordeaux1.fr} 
\keywords{Euclidean lattices, successive minima, bases 
\newline 
(*) Institut de Math\'ematiques, C.N.R.S et Universit\'e {\sc Bordeaux~1}, 
UMR 5251\\ 
{\it{\bf Version 4}, September 28\thh, 2012} 
}

\begin{abstract} 
We compare for an $n$-dimensional Euclidean lattice $\Lb$ 
the smallest possible values of the product of the norms 
of $n$~vectors which either constitute a basis for $\Lb$ 
(Hermite-type inequalities) or are merely assumed to be independent 
(Minkowski-type inequalities). We improve on 1953 results of 
van der Waerden in dimensions $6$ to $8$ and prove partial result 
in dimension~$9$. 
%\vskip.1truecm\noi{\sc R\'esum\'e.} %\newline {\sc Titre fran\c cais~:} 
\end{abstract} 
\maketitle

%%Section 1 
\section{Introduction}\label{secintro} 

We consider a Euclidean space~$E$ of dimension~$n$ and {\em lattices} 
\linebreak 
$\Lb,\Lb',\dots$, that is discrete subgroups of rank~$n$ of $E$. 
For $x\in E$, we define the {\em norm of~$x$} by 
$N(x)=x\cdot x$ (the square of the traditional $\| x\|$). 
The {\em determinant $\det(\Lb)$ of $\Lb$} is the determinant 
of the {\em Gram matrix} $\Gram(e_i\cdot e_j)$ of any basis 
$\cB=(e_1,\dots,e_n)$ for~$\Lb$. 
We also define the {\em minimum of $\Lb$} as 
$\min\Lb=\min_{x\in\Lb\sm\{0\}}\,N(x)$, 
and its {\em Hermite invariant} 
$\g(\Lb)=\frac{\min\Lb}{\det(\Lb)^{1/n}}$. 
The {\em Hermite constant for dimension~$n$} 
%xXx 
is $\g_n=\sup_\Lb\,\g(\Lb)$. 
(Theorem~\ref{thHermMin} below shows that $\g_n$ exists.) 

\smallskip 

For a lattice $\Lb$ in $E$, define $H_b(\Lb)$ and $M(\Lb)$ as 
$$\frac{\min N(e_1)\cdots N(e_n)}{\det(\Lb)}$$ 
on bases $(e_1,\dots,e_n)$ for $\Lb$, and independent vectors 
of~$\Lb$, respectively. 
\newline 
Set 

\ctl{$Q_b(\Lb)=\dfrac{H_b(\Lb)}{M(\Lb)}\,.$} 

\smallskip\noi 
Hermite, in a series of letters to Jacobi, then Minkowski, in his book 
{\em Geometrie der Zahlen\/}, obtained the following bounds: 

%%Theorem 1.1 
\begin{theorem}\label{thHermMin} For any $n$-dimensional 
lattice $\Lb$, we have 
$$H_b(\Lb)\le \Big(\frac 43\Big)^{n(n-1)/2}\ \nd\ 
M(\Lb)\le\g_n^n\,.$$ 
\end{theorem} 

\noi{\rm\small[Note that (using an argument of density) 
Minkowski proved a linear bound for $\g_n$ whereas Hermite's 
(derived from the bound for $H_b(\Lb)$) is exponential.]} 

Proofs of the theorem above can be read in \cite{M}, 
Theorems~2.2.1 and~2.6.8. The 
%(non-classical) 
proof of Minkowski's 
theorem given there makes use of a deformation trick, useful 
in our context: one proves that the local maxima of $M$ 
are attained on {\em well-rounded lattices}, 
that is lattices $L$ having $n$~independent minimal vectors, 
so that $M(L)=\g(L)^n$ is bounded from above by $\g_n^n$. 
We may of course chose $e_1$ minimal among non-zero vectors, 
then $e_2$ minimal among vectors not proportional to $e_1$, etc., 
whence the name of {\em theorem of successive minima} generally 
given to Minkowski's theorem. 

\medskip 

It is well known (and we shall recover this fact below) 
that for a lattice $\Lb$ of dimension $n\le 4$, successive minima 
constitute a basis for $\Lb$ except possibly if $\Lb$ is the 
$4$-dimensional centred cubic lattice, for which index~$2$ 
may occur. Since this lattice possesses a basis of minimal vectors, 
we have $M(\Lb)=H_b(\Lb)$ up to dimension~$4$. This is no longer true 
for $n>4$, as shown by centred cubic lattices. 

\smallskip 

In his 1953 {\em Acta Mathematica} paper \cite{vdW}, van der Waerden 
gives a recursive formula for a bound for $\frac{H_b}M$ 
in dimensions $n\ge 4$. In a visit to Bordeaux (October, 2008), 
Achill Sch\"urmann pointed out to me that van der Waerden's formula 
may be given the ``closed'' form below: 

%%Theorem 1.2 
\begin{theorem}\label{thvdW} For $n\ge 4$, we have 
$Q_b(\Lb)\le\Big(\dfrac 54\Big)^{n-4}\,.$ 
\end{theorem} 

He also put forward the conjecture 
(based on properties of the Voro\-noi cones) 
that the bound $\frac n4$ could hold for $4\le n\le 8$, 
a better bound than van der Waerden's for $n=6,7,8$. 
This is the main theorem we are going to prove. 

%%Theorem 1.3 
\begin{theorem}\label{thn<=8} For $4\le n\le 8$, we have 
$Q_b(\Lb)\le\frac n4$, and equality is needed \ifff $\Lb$ 
is a centred cubic lattice. 
\end{theorem} 

The choice of an orthonormal basis identifies $E$ with $\R^n$ 
equipped with his canonical basis $\cB=(e_1,\dots,e_n)$, 
which generates the lattice $\Z^n$. Centred cubic lattices 
are the lattices which are similar to 

\smallskip\ctl{$C_n:=\la\cB,e\ra$ \ where \ 
$e=\frac{e_1+\dots+e_n}2$\,,} 

\smallskip\noi 
which can be viewed as a lift of the (unique) $[n,1,n]$-binary code. 
We can define similarly the canonical lift $\Lb_C$ 
of any binary code~$C$ of weight $\wt(C)\ge 4$, obtaining this way 
a lattice $\Lb_C$ of minimum~$1$. Taking for $C$ the unique 
$[9,2,6]$-(binary) code $C_9$, 
with generator matrix 

\smallskip\ctl{$\left(\stx 
1&1&1&1&1&1&0&0&0\\0&0&0&1&1&1&1&1&1 \estx\right)$} 

%\smallskip 
\noi 
and weight distribution $(6^3)$, we again obtain a lattice 
with 

%\smallskip 
\ctl{
$Q_b(\Lb)=\Big(\frac 64\Big)^2=\frac 94$\,.} 

\smallskip\noi 
%xXx 
This shows that the statement of Theorem~\ref{thn<=8}
does not extend as it stands in dimensions~$n\ge 9$. 

\smallskip

The proofs of Theorem~\ref{thn<=8} for certain codes 
that I give below are often valid beyond dimension~$8$. 
This suggests that the bound $\frac n4$ is still valid 
for $n=9$. This is the conjecture below, 
which I partially prove in the next theorem. 
However in order that this paper should not have an unreasonable 
length I did not try to prove all cases; 
see Proposition~\ref{thn=9} and Remark~\ref{remindex42}. 

%%Conjecture 1.4 
\begin{conj} \label{conjn=9} 
For $n=9$, we have $Q_b(\Lb)\le\frac 94$, and equality is needed 
\ifff $\Lb$ is either a centred cubic lattice or is similar 
to the canonical lift of $C_9$. 
\end{conj} 

%%Theorem 1.5 
\begin{theorem}\label{thn=9} 
Conjecture~\ref{conjn=9} is true 
%if either $\Lb$ is well rounded, or 
if $\Lb$ contains a sublattice $\Lb'$ generated by a frame 
of successive minima for $\Lb$ which satisfies one of the following 
conditions: 
\begin{enumerate} 
\item 
$\Lb/\Lb'$ is $2$- or $3$-elementary. 
\item 
$[\Lb:\Lb']=4$. % or $6$. 
\item 
$[\Lb:\Lb']\ge 9$. 
\end{enumerate} 
\end{theorem} 

Enlarging 
the code 
$C_9$ with a column 
{\small $\left(\stx 0\\1\estx\right)$}, 
we obtain the unique 
{\em odd} $[10,2,6]$-binary code $C_{10}$; this has weight 
distribution $(6\cdot 7^2)$, and its lift has 
$Q_b=\frac 64\cdot\frac 74=\frac{21}8>\frac{10}4$; 
and lifting convenient binary codes of length~$n$ and dimension~$2$ 
indeed suffices to show that the bound $\frac n4$ no longer holds 
beyond~$n=9$. 

\bigskip 

Here is an outline of the method used to prove Theorems 
\ref{thn<=8} and~\ref{thn=9}. 
For every lattice $\Lb\subset E$, we denote by $\Lb'$ a lattice 
having as a basis a frame $(e_1,\dots,e_n)$ of successive minima 
for $\Lb$ and by $d$ the annihilator of $\Lb/\Lb'$. 
We define the {\em maximal index $\imath(\Lb)$ of $\Lb$} 
as the maximal value of the index $[\Lb:\Lb']$ for $\Lb'$ as above. 

Given $d$ we may write 

%\smallskip 
\ctl{$\Lb=\la\Lb',f_1,\dots,f_k\ra$} 

 \smallskip\noi 
for vectors $f_i$ of the form $f=\frac{a_1 e_1+\dots a_n e_n}d\,.$ 
The collection of the $n$-tuples $(a_1,\dots,a_n)$ modulo~$d$ 
defines a $\Z/d\Z$-code canonically associated with $(\Lb,\Lb')$. 
These codes are classified for $n\le 8$ in \cite{M1}, 
where I extended previous work by Watson, Ryshkov and Zahareva; 
\cite{Wa}, \cite{Ry}, \cite{Za}), and for $n=9$ in \cite{K-M-S}. 
The proof of Theorems \ref{thn<=8} and \ref{thn=9} heavily relies 
on the classification of these $\Z/d\Z$-codes 
(though some general inequalities will sometimes allow us to skip 
a detailed case-by-case analysis): we shall calculate for each 
admissible code $C$ an upper bound of $\frac{H_b(\Lb)}{M(\Lb)}$ 
for $\Lb\in C$ and check that $\frac n4$ is attained only on codes 
defining the lattices listed in these theorems. 

\smallskip 

The bounds we shall prove for a given code are scarcely optimal, 
and a closer look will show that they are not optimal 
whenever they are not sharp on well-rounded lattices. 
Probably the exact bounds on all codes are attained 
only on well-rounded lattices. 
An {\em a priori} proof of this result would considerably 
simplify our proofs. %(and prove conjecture~\ref{conjn=9}). 

\smallskip 

It should be noted that the results of \cite{M1} were obtained 
essentially by hand: we made use of a computer only to prove 
the existence of some particular codes, which does not matter 
for this paper. So Theorem~\ref{thn<=8} will be proved 
within the frame of ``classical'' mathematics. 

This is no longer true for dimension~$9$. 
Though classification details of $\Z/d\Z$-codes can (or could) 
be skipped for small values of $d$, I do not see any way 
of avoiding the heavy calculations using linear programming 
packages performed in \cite{K-M-S} to prove that 
only index $16$ need be considered if $[\Lb:\Lb']>12$. 
This problem of large indices shows up from dimension $7$ onwards. 
In \cite{Wa}, Watson proved that if $n=7$ and $[\Lb:\Lb']>5$, 
then $\Lb\sim\E_7$ (and $\Lb/\Lb'$ is $2$-elementary), 
and stated an analogue for dimension~$8$, 
for which a proof can be read on my home page: 
if $n=8$ and $[\Lb:\Lb']>8$, then $\Lb\sim\E_8$ 
(and $\Lb/\Lb'$ is elementary of order $3^2$ or $2^4$). 

\medskip 

After having recalled in Section~\ref{secbackground} 
some general facts on Watson's index theory, 
we establish in Section~\ref{secdenom2} 
sharp bounds for $Q_b(\Lb)$ when $\Lb/\Lb'$ is $2$-elementary. 
Then Section~\ref{secdim<=7} is devoted to dimensions~$n\le 7$, 
and index~$3$, and to some cases of index~$4$. 
Dimension~$8$ is dealt with in Section~\ref{secindex4-5} 
after having proved complements on index~$4$. 
This will complete the proof of Theorem~\ref{thn<=8}. 
Theorem~\ref{thn=9} is then proved in Section~\ref{secdim9}. 
%and Section~\ref{seccompl} is devoted to various complements. 

\medskip 

Actually the reference to a basis in the definition of $H_b$ 
is not pertinent: in \cite{M-S} is displayed an example 
of a $10$-dimensional lattice $L$ which is generated by its minimal 
vectors but has no basis of minimal vectors, 
so that $Q_b(L)$ is strictly larger than one, 
though it would be reasonable to consider that successive minima 
suffice to describe the behaviour of~$L$. 

We may define as follows an invariant $H_g$ %($g$ for ``generators'') 
for a lattice~$\Lb$. 

\noi{\em 
For every finite set $\cG$ of generators of $\Lb$, take the maximum 
$M_\cG(\Lb)$ of the products $N(e_1)\cdots N(e_n)$ on all systems 
of independent vectors $e_1,\dots,e_n$ extracted from $\cG$, 
and define $H_g(\Lb)$ as the lower bound (indeed, a minimum) 
of $M_\cG(\Lb)$ on all generating sets $\cG$. 

Finally let $Q_g(\Lb)=\frac{H_g(\Lb)}{M(\Lb)}$}. 

\smallskip 

\noi We clearly have $H_g(L)=M(L)$ for the lattice $L$ above. 
It will turn out that for dimensions $n\le 8$ the exact bounds 
for $Q_b$ of Theorem~\ref{thn<=8} is also the exact bounds for $Q_g$ 
(and also for $n=9$ under Conjecture~\ref{conjn=9}).

%%Section 2 
\section{Some background}\label{secbackground} 

The basic methods and results on Watson's {\em index theory} 
can be read in \cite{M1} and \cite{K-M-S}. Here we recall a few facts 
that will be used all along this paper, beginning with 
{\sl Watson's identity}, the most fruitful tool for what follows, 
the (simple) proof of which is left to the reader.

\subsection{Watson's identity} 
%%Proposition 2.1 
\begin{prop} \label{propWatsonindentity} {\rm(Watson)} 
Let $\cB=(e_1,\dots,e_n)$ be a basis for $E$ and let $a_1,\dots,a_n$ 
and $d>1$ be integers. For $\lb\in\R$, let $\sgn(\lb)=-1$, $0$ or $1$ 
according as $\lb$ is negative, positive or zero. 
Let 
%\smallskip\ctl{ 
$e=\frac{a_1 e_1+\dots+a_n e_n}d$\,.%}% 
%\smallskip\noi 
Then 
$$\sum|a_i|\big(N(e-\sgn(a_i)\,e_i)-N(e_i)\big)= 
\left(\Big(\sum_{i=1}^n|a_i|\Big)-2d\right)N(e)\,.\qed$$ 
%In particular we have $\sum_{i=1}^n|a_i|\ge 2d$ 
%if $(e_1,\dots,e_n)$ is a frame of successive minima 
%for $\Lb=\la e_1,\dots,e_n,e\ra$. 

\end{prop} 

%%Definition 2.2. 
\begin{defi} \label{defm_i} {\rm 
In the sequel we denote by $\cB=(e_1,\dots,e_n)$ a basis for $E$ 
and by $\Lb'$ the lattice it generates. With the data above, 
we set $A=\sum_j\,\abs{a_j}$. For $i\ge 0$ we denote by $S_i$ 
the set of subscripts $j$ (or of vectors $e_j$) 
for which $\abs{a_j}=i$ and set $m_i=\abs{S_i}$, 
and define $m\le n$ by $m=\sum_{i\ne 0}\,m_i$. 
We also set $T=\dfrac{e_1+\dots+e_n}d$.} 
%Check that this bad notation is not used in the sequel !!! -- OK !! 

{\rm We say that {\em Watson's condition holds} if $A=2d$ 
and the $a_i$ are non-zero 
%\linebreak 
(i.e., if $A=2d$ and $m=n$).} 
\end{defi} 

%%Proposition 2.3 
\begin{prop} \label{propWatsoncond} 
Assume that Watson's condition holds. Then: 
\begin{enumerate}
\item 
We have $N(e-\sgn(a_i)e_i)=N(e_i)$ for all $i$. 
\item 
We have $\abs{a_i}\le\frac d2$ for all $i$. 
\item 
If $(e_1,\dots,e_n)$ is a frame of successive 
%\linebreak 
minima for 

\smallskip\ctl{ 
$\Lb:=\la\Lb',e\ra\,=\ \cup_{k\!\!\mod d}\ k e+\Lb'$\,,} 

\smallskip\noi 
the $e_i$ have equal norms. 
\item 
If moreover $m_1\ge 1$, then $H_b(\Lb)=M(\Lb)$. 
\end{enumerate} 
\end{prop} 

\begin{proof} 
Negating some $e_i$ if need be, we may assume 
that all $a_i$ are positive.

(1) Since the right hand side in Watson's identity is zero, 
we have $a_i(N(e-e_i)-N(e_i))=0$ for all $i$. 

(2) If $a_i$ is larger than $\frac d2$ for some~$i$, 
then replacing $e$ by $e-e_i$ in Watson's identity changes 
$A$ into $A+(d-2 a_i)<2d$. 

(3) Suppose that $N(e_i)<N(e_{i+1})$ for some~$i$. 
By (2), replacing $e_{i+1}$ by $e-e_i$, we still have 
a system of independent vectors, with $N(e-e_i)=N(e_i)$ by~(1), 
which contradicts the fact that $(e_1,\dots,e_n)$ is a frame 
of successive minima. 

(4) Choose $i$ with $a_i=1$. 
Then replacing $e_i$ by $e-e_j$ for some $j\ne i$, 
we obtain a basis for $\Lb$ made of vectors of norm $\min\Lb'$. 
\newline{\small 
[Note that the equality $H_g(\Lb)=M(\Lb)$ holds even if $m_1=0$.]} 
\end{proof} 

%xXx Former subsection 2.2 ``Lifting bases'' ; see vxherm-min.tex 

\subsection{A crude bound} % Now 2.2, formerly 2.3 
We consider a frame $\cB=(e_1,\dots,e_n)$ of successive minima 
for a lattice $\Lb$, denoting by $\Lb'$ 
the lattice with basis $\cB$, and assume that $\Lb/\Lb'$ is cyclic 
of order~$d$, writing $\Lb=\la\Lb',e\ra$ 
with $e=\frac{a_1 e_1+\dots a_n e_n}d$. 
Reducing modulo $d$ the numerator of $e$ and negating some $e_i$, 
we may and shall assume 
that the $a_i$ satisfy $0\le a_i\le\frac d2$. 

%%Proposition 2.4 
\begin{prop} \label{propcrudebound} 
With the hypotheses above, we have 
$$ N(e)\le\frac{\sum_{i=1}^n\,a_i N(e_i)
\sum_{j=i}^n\,a_j}{d^2}\,,$$ 
and in particular, 
$$ N(e)\le\frac{\sum_i\,m_i(m_i+1)/2\cdot i^2+% 
\sum_{i<j}\,m_i m_j\cdot i j}{d^2}\,N(e_n)\,.$$ 
\end{prop} 

\begin{proof} 
Just develop the expression of $e$, and observe that if $i<j$ 
(because the $e_i$ are successive minima), 
we have $N(e_j-e_i)\ge N(e_j)$, hence 

\smallskip\ctl{ 
$2\,e_i\cdot e_j=N(e_i)+N(e_j)-N(e_j-e_i)\le N(e_i)$} 
\end{proof} 

We shall use this crude bound to bound the norm of vectors 
\linebreak 
$e-e_i$ or $e-e_i-e_j$ by successive applications of Watson's 
identity, and also sometimes prove improvements for a convenient 
choice of $e$, as in Lemma~\ref{lemindex2} below. 
We quote as a corollary the case of equal $a_i$, 
the proof of which is and easy consequence of the inequalities 

\smallskip\ctl{
$(n-k+1)N(e_k)+k N(e_{n-k+1})\le\frac{n+1}2\,(N(e_k)+N(e_{n-k+1})$\,.} 

\noi for $k=1,\dots,\lf\frac n2\rf$\,: 
%%Corollary 2.5 
\begin{corol} \label{corcrudebound} 
If $e=\dfrac{e_1+\dots+e_n}d$, then 
$$ N(e)\le\frac{n+1}{2d^2}\,\sum_{i=1}^n\,N(e_i)\le 
\frac{n(n+1)}{2d^2}\,N(e_n)\,.\qed$$ 
\end{corol} 

When constructing bases for $\Lb$ from a frame $\cB$ of successive 
minima, we shall replace some vectors $e_i$ of $\cB$, including $e_n$, 
by convenient vectors $f_i\in\Lb\sm\Lb'$. 
We shall then have to bound a product $\prod\frac{N(f_i)}{N(e_i)}$, 
where in practice, $i$ is the largest subscript in the support 
of the numerator of~$f_i$. 
%It will finally turn out that our inequalities rely 
%on the inequalities $N(e_i)\le N(e_n)$, 
%thus reduce to the case when all $e_i$ have equal norms. 

\smallskip 

Our results solely depend on the similarity class 
of~$\Lb$. For these reason we shall often assume 
from Section~\ref{secdim<=7} onwards that $\Lb$ is scaled 
so that $N(e_n)=1$.

%%Section 3 
\section{2-elementary quotients} 
\label{secdenom2} 

In this section we apply the theory of binary codes to 
obtain bounds for $H_b(\Lb)/M(\Lb)$ when $\Lb/\Lb'$ 
is $2$-elementary. The results we obtain together with those 
of Section~\ref{secbackground} suffice to prove 
Theorem~\ref{thHermMin} in dimensions $n\le 6$. 

\subsection{Index 2} \label{subsec3.1} 
In this subsection we assume that $\Lb=\la\Lb',e\ra$ 
where $e=\dfrac{e_1+\dots+e_n}2$. 

%%Lemma 3.1 
\begin{lemma} \label{lemindex2} 
Let $\cS=\Big\{\dfrac{e_1\pm e_2\dots\pm e_n}2\Big\}$. 
\begin{enumerate} 
\item %(1) 
There exists $x\in\cS$ of norm 
$N(x)\le\dfrac{N(e_1)+\dots+N(e_n)}4$, 
and equality is needed \ifff the $e_i$ are pairwise orthogonal. 
\item %(12 
If all vectors in $\cS$ have a norm $N\ge\max\,N(e_i)$, 
then we have $N(x)\le\frac n4\,\max N(e_i)$, 
and equality holds \ifff the $e_i$ have equal norms and 
$\Lb$ is the centred cubic lattice constructed on the~$e_i$. 
\end{enumerate} 
\end{lemma} 

\begin{proof} 
Negating some $e_i$ if need be, we may assume that $e$ 
is the shortest of the vectors $\frac{e_1\pm e_2\dots\pm e_n}2$. 
We thus have $N(e-e_i)\ge N(e)$ for $i=1,\dots,n$. 
Summing on $i$ and applying Watson's identity for~$e$, we obtain 

\medskip\ctl{\small 
$n N(e)\le\sum_i\big(N(e-e_i)-N(e_i)\big)+\sum_i N(e_i)
=(n-4)N(e)+\sum_i N(e_i)$\,,} 

\medskip\noi 
i.e., $N(e)\le\frac{N(e_1)+\dots+N(e_n)}4$. 

\medskip 

If equality holds, we must have $N(e-e_i)=N(e)$ for all~$i$. 
Watson's identity for $e-e_i$, which reads 

\medskip\ctl{\small 
$\big(N(e)-N(e_i)\big)+\sum_{j\ne i}\big(N(e-e_i-e_j)-N(e_j)\big)= 
(n-4)N(e-e_i)$\,,\quad$(\ast)$} 

\medskip\noi 
implies $\sum_{j\ne i}\,N(e-e_i-e_j)=(n-1)N(e)$. 
Since $N(e-e_i-e_j)\ge N(e)$, 
the $(n-1)$ terms $N(e-e_i-e_j)$ must be equal to $N(e)$ 
for all distinct subscripts $i,j$. 
%\smallskip 
The identity 

\smallskip\ctl{ 
$N(e-e_i-e_j)+N(e)=N(e-e_i)+N(e-e_j)+2e_i\cdot e_j$} 

\smallskip\noi 
then shows that all scalar products $e_i\cdot e_j$ must be zero. 

The converse is clear. This completes the proof of (1). 

\medskip 

%\smallskip\noi 
%Summing on $i$ we obtain 

%\begin{gather*} 
%n N(e)-\sum_i N(e-e_i)+ 
%\sum_{j\ne i}\big(N(e-e_i-e_j)-(n-1)\sum_j N(e_j)\big)\\= 
%(n-4)\sum_i N(e-e_i)\,.\qquad\qquad(\ast\ast)
%\end{gather*} 

%If equality holds in Lemma~\ref{lemindex2}, we must have $N(e-e_i)=N(e)$ for all~$i$, and the identity above implies 

%\medskip\ctl{ 
%$\sum_{j\ne i} N(e-e_i-e_j)=(n^2-n)N(e)$\,.} 

%\smallskip\noi 
%Since the left hand side is a sum of $n^2-n$ terms, each of which 
%grater than or equal to $N(e)$, equality holds \ifff 
%we have $N(e-e_i-e_j)=N(e)$ for all $i$ and all $j\ne i$. 

Still assuming that $e$ has the smallest norm on $\cS$, we may assume 
that we have $N(e_1)\le\dots\le N(e_n)$. We then clearly have 
$$\dfrac{N(e_1)+\dots+N(e_n)}4\le \frac n4\,N(e_n)\,,$$ 
and the inequality is strict unless all $e_i$ have the same norm as~$e_n$. 
Then $\Lb$ is a centred cubic lattice, and conversely 
centred cubic lattices satisfy $N(x)=\frac n4\,N(e_i)$ for all $x\in\cS$. 
\end{proof} 

%%Corollary 3.2 
\begin{corol} \label{corindex2} 
If a lattice $\Lb$ contains to index $\imath=2$ a sublattice $\Lb'$ 
generated by successive minima of $\Lb$, 
then $\dfrac{H_b(\Lb)}{M(\Lb)}$ is bounded from above by $\frac n4$, 
and equality holds \ifff $\Lb$ is a centred cubic lattice. 
\end{corol} 

\begin{proof} 
Just apply Lemma~\ref{lemindex2} to a frame of successive minima 
$e_1,\dots,e_n$ for $\Lb$ generating a lattice of index~$2$ in~$\Lb$: 
$(e_1,\dots,e_{n-1},e)$ is then a basis for~$\Lb$, so that 
$\dfrac{H_b(\Lb)}{M(\Lb)}=\dfrac{N(e)}{N(e_n)}$. 
\end{proof}

%%Remark 3.3 
\begin{remark} \label{remdim4} {\small\rm 
Formula $(\ast)$ above shows that when $n=4$, 
all vectors $e_i$, $e$, $e-e_i$, $e-e_i-e_j$ have the same norm. 
The remaining of the proof of Lemma~\ref{lemindex2} 
then shows that $\Lb$ must be a centred cubic lattice.} 
\end{remark} 

%%Remark 3.4 
\begin{remark} \label{remdenom4} {\rm(Watson)} {\small\rm 
Let $\Lb/\Lb'$ be cyclic of order~$4$, with $\Lb=\la\Lb',e\ra$, 

\smallskip\ctl{ 
$e=\frac{e_1+\dots+e_{m_1}+2(e_{m_1+1}+\dots+e_{m_1+m_2})}4= 
\frac{e'+e_{m_1+1}+\dots+e_{m_1+m_2}}2\,,$} 

\smallskip\noi 
$e'=\frac{e_1+\dots+e_{m_1}}2$. 
Then Watson's identity shows that $m_1>4$ implies $n\ge 7$, 
and Remark~\ref{remdim4} shows that if $m_1=4$, 
then we must have $m_2\ge 3$, hence again $n\ge 7$, 
and that if $m_1=4$ and $n=7$, then $e$ is minimal.
This last conclusion holds more generally under Watson's condition 
if some coefficient $a_i$ is equal to $\frac d2$, since we may then 
apply Watson's identity to $e'=e-e_i$ instead of~$e$.} 
\end{remark} 

%%Subsection 3.2 
\subsection{Binary codes and 2-elementary quotients} \label{subsec2elem} 
In this subsection we consider a pair of lattices $\Lb$ and 
$\Lb'\subset\Lb$ such that $\Lb/\Lb'$ is $2$-elementary 
of order $2^k$ %$\ge 2$ 
and $\min\Lb=\min\Lb'$. We choose a basis 
$\cB=(e_1,\dots,e_n)$ for $\Lb'$ and denote by $C$ 
the binary code (of length~$n$ and dimension~$k$) 
defined by $(\Lb,\Lb',\cB)$. Since $\min\Lb=\min\Lb'$, 
$C$ has weight $w\ge 4$. 

%%Proposition 3.5 
\begin{prop} \label{prop2elem} 
Assume that $\cB$ is a frame of successive minima for~$\Lb$. 
\begin{enumerate} 
\item 
We have $\dfrac{H_g(\Lb)}{M(\Lb)}\le 
\dfrac{\min\{\wt(\a_1)\cdots\wt(\a_k)\}} {4^k}\,,$ 
where the minimum is taken over all bases 
$\a_1,\dots,\a_k$ for $C$. 
\item 
Assume that $C$ is irreducible (which implies that 
the support of $C$ is the whole set $\{1,\dots,n\}$). 
Then if equality holds in (1), 
the $e_i$ have equal norms. 
%the $e_i$ are pairwise orthogonal and have equal norms. 
\item 
If $k\le 2$ the conclusions of (1) and (2) hold for 
$\dfrac{H_b(\Lb)}{M(\Lb)}$. 
\end{enumerate} 
\end{prop} 

\begin{proof} 
(1) By Corollary~\ref{corindex2} we can lift each word $\a_i$ 
to a vector $x_i\in\Lb$ of norm $N(x_i)\le\frac{\wt(\a_i)}4$. 
%and we also have $N(x_i)\ge N(e_n)$ since $(e_i)$ is a frame 
%of successive minima for~$\Lb$. 
A set $S$ of $n$ independent vectors extracted from the set 
$\{x_i,e_i\}$ consists of $\ell\le k$ vectors $x_i$ 
and $n\!-\!\ell$ vectors $e_j$, satisfying the condition: 
for every $i$ there exists $j=j[i]$ in the support of $\a_i$ 
such that $e_j$ is not in~$S$. It is then clear 
that we have 
$$ 
\frac{\prod_{x\in S}\,N(x)}{\prod_{1\le i\le n}\,N(e_i)}\le 
\prod_{i=1}^\ell\,\frac{N(x_i)}{N(e_{j[i]})}\le 
\prod_{i=1}^\ell\,\frac{wt(\a_i)}4\le 
\prod_{i=1}^k\,\frac{wt(\a_i)}4\,. 
$$ 
(The last inequality results from the lower bounds $\wt(\a_i)\ge 4$, 
which hold because $\min\Lb=\min\Lb'$.) 
%This proves (1). 
%Then $\{e_1,\dots,e_n,x_1,\dots,x_k\}$ generates $\Lb$, 
%and we have $N(x_i)\le \frac{wt(\a_i)}4\,N(e_n)$. 

\smallskip 

(2) Since $C$ is irreducible, we may order $\a_1,\dots,a_k$ 
so that the supports of $\a_i$ and $\a_{i+1}$ have a non-empty 
intersection for every $i<k$. By Lemma~\ref{lemindex2}, 
we have $N(e_i)=N(e_j)$ whenever $i,j$ both belong to the support 
of some $\a_\ell$, and the hypothesis 
$\Supp(\a_i)\cap\Supp(\a_{i+1})\ne\emptyset$ proves (2). 

\smallskip 

(3) If $k=1$, we obtain a basis for $\Lb$ by replacing any $e_i$ 
with $i\in\Supp(x_1)$ by $x_1$. This method clearly extends 
(by induction) to all codes satisfying the condition 

\smallskip\ctl{ 
$\forall\,i,\,\Supp(\a_i)\not\subset\cup_{j\ne i}\,\Supp(\a_j)\,.$} 

\smallskip\noi
This remark applies in particular to codes of dimension~$2$, 
for if there were an inclusion, say, 
$\Supp(\a_2)\subset\Supp(\a_1)$, 
we could replace $\a_1$ by $\a_1+\a_2$, a word of smaller weight. 
\end{proof} 

%%Remark 3.6 
\begin{remark} \label{remparameters} {\small\rm 
Two vectors $e_i$ and $e_j$ are necessarily orthogonal if $i,j$ belong 
to the support of some $\a_\ell$ (or of some word of weight~$4$), 
but this is not general. For instance, if $C$ is the code 
$[8,2,5]$-code 

\smallskip\ctl{ 
$C_8=\left(\stx 1&1&1&1&1&0&0&0\\0&0&0&1&1&1&1&1\estx\right)$} 

\smallskip\noi 
(see Subsection~\ref{subsecden2n<=10} below), 
the $e_i$ must have equal norm and be pairwise orthogonal 
if $i<j\le 5$ or $4\le i<j$, but we still have 
$\frac{H_b}M=\frac{25}{16}$ on the lifts of $C_8$ provided that 
$\abs{e_i\cdot e_j}$ be small enough for $i=1,2,3$ and $j=6,7,8$ 
so as to have $N(x)\ge \frac 54$ for any $x\in\Lb$ 
which lifts the weight-$6$ word $(1^3\,0^2\,1^3)$. 
Thus the lattices $\Lb$ which lift $C_8$ in a given scale 
(say, $\min\Lb=1$) depend on $9$~parameters.} 
\end{remark}

%%Subsection 3.3 
\subsection{Dimensions up to 10} \label{subsecden2n<=10} 
We consider an $[n\le 10,k\ge 1,w\ge 4]$-binary code $C$. 
We prove for quotients $\Lb/\Lb'$ associated with $C$ the bounds 
for $\frac{H_b(\Lb)}{M(\Lb)}$ announced in the introduction 
($\frac n4$ if $4\le n\le 9$, $\frac{21}8$ if $n=10$), and
characterize the cases when equality holds. 

\smallskip 

We may assume that $k\ge 2$ (since the case when $k=1$ has been dealt 
with in Corollary~\ref{corindex2}), that the support of $C$ is 
the whole set $\{1,\dots,n\}$ 
(since otherwise we may apply results for dimension $n-1$), 
and that $w>4$ (if $wt(\a_1)=4$, we reduce ourselves to the case 
of $\dim C=k-1$ by considering $\la\Lb',x_1\ra$ instead of $\Lb'$). 
Then $C$ contains an even subcode $C_0$ of dimension 
$k-1$ and weight $w_0\ge 6$. 
(Note however that $\abs{\Supp(C_0)}$ 
may be strictly smaller than~$n$.) 

\smallskip 

It is readily verified that for $n\le 8$, every $[n,2,w\ge 4]$-code 
contains a word of weight~$4$, except for the a unique $[8,2,5]$-code 
(the code $C_8$ of Remark~\ref{remparameters}). 
This has weight distribution 
$6\cdot 5^2$, so that its lifts satisfy 
$\frac{H_b(\Lb)}{M(\Lb)}\le\frac{25}{16}<\frac n4=2$. 
This also proves the existence of a weight-$4$ word if $k\ge3$, 
and completes the proof of Theorem~\ref{thn<=8} 
for $2$-elementary quotients. 

\smallskip 

Let now $n=9$ and first $k=2$. It is again readily verified 
that codes of weight $w\ge 5$ and support $\{1,\dots,9\}$ 
have weight distributions 
$8\cdot 5^2$, $6\cdot 5\cdot 7$ or $6^4$ 
and that there exists a unique code for each weight distribution, 
which gives for $\frac{H_b(\Lb)}{M(\Lb)}$ the exact bounds 
$\frac{25}{16}$, $\frac{15}8$ and $\frac 94$, respectively, 
and proves that if $k=3$, $C$ must extend the code $C_9$. 
It is then easily checked that such an extension 
by a word of weight~$5$ (resp.~$6$) must contain a word 
of weight~$3$ (resp.~$4$). 

This completes the proof of Theorem~\ref{thn=9} 
for $2$-elementary quotients. 

\smallskip 

Let now $n=10$ and first $k=2$. We easily check as above 
that codes of weight $w\ge 5$ and support $\{1,\dots,10\}$ 
have weight distributions 
$10\cdot 5^2$, $8\cdot 5\cdot 7$, $6\cdot 5\cdot 9$,
$6\cdot 7^2$, and $6^2\cdot 8$. 
The largest upper bound for $\frac{H_b(\Lb)}{M(\Lb)}$ 
is $\frac{21}8$, attained on a unique code, namely 

\smallskip\ctl{ 
$C_{10}=\left(\stx 1&1&1&1&1&1&0&0&0&0\\ 
0&0&0&0&0&1&1&1&1&1\estx\right)$\,.} 

\smallskip\noi 
This also shows that there are exactly two even $[10,2,w\ge 6]$-codes, 
namely 

\smallskip\ctl{ 
$C_{10a}=\left(\stx 1&1&1&1&1&1&0&0&0&0\\ 
0&0&0&0&1&1&1&1&1&0\estx\right)\ \nd\ 
C_{10b}=\left(\stx 1&1&1&1&1&1&0&0&0&0\\ 
0&0&0&0&1&1&1&1&1&1\estx\right)
$} 

\smallskip\noi 
($C_{10a}$ extends $C_9$). It is an easy exercise to check that 
even extensions to $k=3$ of these codes have weight at most~$4$, 
and that each of these codes has a unique odd extension, of weight~$5$. 
We obtain this way two $[9,3,5]$-codes, with weight distributions 
$6^3\cdot 5^3\cdot 7$ and $6^2\cdot 8\cdot 5^4$, 
so that any lift $\Lb$ of one of these codes satisfies the bound

\smallskip\ctl{ 
$\dfrac{H_b(\Lb)}{M(\Lb)}\le\dfrac{125}{64}<\dfrac{21}8$\,.} 

\smallskip 

We state below as a proposition our result for dimension~$10$. 

%%Proposition 3.7 
\begin{prop} \label{prop2elemdim10} 
Let $\Lb$ be a $10$-dimensional lattice having a frame of successive 
minima generating a lattice $\Lb'\!$ such that $\Lb/\!\Lb'$ 
is $2$-elementary. Then we have 

\ctl{$\dfrac{H_b(\Lb)}{M(\Lb)}\le\dfrac{21}8$\,,} 

\smallskip\noi 
and if equality holds, $\Lb$ is a lift of the code $C_{10}$. \qed 
\end{prop}

%%Subsection 3.4 
\subsection{More on index 2} 
We return to the notation of the first subsection, 
Lemma~\ref{lemindex2}, (2), 
but now want for further use to bound the norm of $e$ itself 
rather than that of some suitably chosen vector in $e+\Lb$. 
We write $\Lb=\la\Lb',e\ra$ with $e=\frac{e_1+\dots+e_n}2$, 
and observe that since $\cB=(e_1,\dots,e_n)$ is a frame of successive 
minima for $\Lb$, all vectors $e$, $e-e_i$, $e-e_i-e_j$, etc, 
have a norm larger that $\max_i\,N(e_i)$. 

\smallskip 

%In the proof of the proposition below, 
In this subsection we shall have to consider 
the Coxeter lattices $\A_5^3$ and $\D_6^+$ and the Coxeter-Barnes 
lattices $\A_n^2$, $n\ge 7$, for the definitions of which we refer 
to \cite{M}, Sections 4.4, 5.1 and~5.2. 
Note that $\A_5^3$ and $\A_n^2$ ($n\ge 7$) are perfect whereas 
$\D_6^+$ is not. 

%%Notation 3.8 
\begin{notation} \label{nottuvw} 
Set $t=N(e)$, fix a subscript $i$ which minimizes $u:=N(e-e_i)$, 
then a subscript $j\ne i$ which minimizes $v:=N(e-e_i-e_j)$, 
and finally a subscript $k$ which minimizes $w:=N(e-e_i-e_j-e_k)$. 
\end{notation} 

%%Lemma 3.9 
\begin{lemma} \label{lemtuvw} 
\begin{enumerate} 
\item 
We have the inequalities 

(a) $u\le\frac{n+(n-4)t}n$\,; 

(b) $v\le\frac{n+(n-4)u-t}{n-1}\le 
\frac{2n(n-2)+((n-4)^2-n)t}{n(n-1)}$\,; 

(c) $w\le\frac{n+(n-4)v-2u}{n-2}$. 
\item 
\smallskip 
For $i$ as above and any $\ell>0$, we have 

\smallskip 
\ctl{$N(e-\ell e_i)\le\ell u-(\ell-1)t+\ell(\ell-1)$\,.} 
\end{enumerate} \end{lemma} 

\begin{proof} 
The three assertions in (1) result from the Watson identity 
applied to $e$, $e-e_i$ and $e-e_i-e_j$, respectively, 
and (2) from the identity 
$N(e-\ell e_i)=\ell(e-e_i)-(\ell-1)N(e)+(\ell^2-\ell)N(e_i)$. 
\end{proof}

%%Proposition 3.10 
\begin{prop} \label{propnorm1} 
Assume that $\cB=(e_1,\dots,e_n)$ ($n\ge 4$) is a frame 
of successive minima for~$\Lb$. Then we have 
$$N(e)\le 
%\frac{n+1}8\,\sum_{i=1}^n N(e_i) 
\frac{n(n+1)}8\,N(e_n)\,,$$ 
and for $n=4$, $5$ and $6$, we have the better bounds 

\smallskip\ctl{\small 
$N(e)\le N(e_4)$\,, $N(e)\le\frac 52\,N(e_5)$\,, 
and $N(e)\le 4\,N(e_6)$\,, 
%$N(e)\le\frac 14\,\sum_{i=1}^n N(e_i)$,\ 
%$N(e)\le\frac 12\,\sum_{i=1}^n N(e_i)$\,,\ 
%and $N(e)\le\frac 43\,\sum_{i=1}^n N(e_i)$\,, 
} 

\smallskip\noi 
respectively. 
These bounds are optimal and attained uniquely on well-rounded 
lattices. Moreover, if $n\ne 6$, they are attained on a unique 
similarity class of lattices. 
\end{prop} 

\begin{proof} 
Without loss of generality, 
we may assume that $\Lb$ has been rescaled so that $N(e_n)=1$, 
which implies that $T:=N(e_1)+\dots+N(e_n)$ is bounded 
from above by~$n$. We first prove the upper bounds. 

\smallskip 

The first inequality is merely the crude bound 
of Proposition~\ref{propcrudebound}. 

\smallskip 

For $n\le 6$, the coefficient of $t$ in the second inequality 
in (1b) of Lemma~\ref{lemtuvw} is negative, so that the inequality 
$v\ge 1$ implies $t\le 1$ if $n=4$ and $t\le\frac 52$ if $n=5$. 

\smallskip 

The inequalities of Lemma~\ref{lemtuvw} do not suffice to prove 
the proposition if $n=6$. To deal with this case we use directly 
the Watson identities relative to $e$, to the $e-e_i$, 
and to $e-e_i-e_j$, namely 

{\small\begin{gather*} 
(a)\quad\sum_{i=1}^6\,N(e-e_i)=T+2N(e)\,,\\ 
\text{then } \forall\,i\,,\\ 
(b)\quad N(e)+\sum_{j\ne i}\,N(e-e_i-e_j)= 
T+2N(e-e_i)\,,\\ 
\nd\forall\,i,j\,,\\ 
(c)\ N(e-e_i)\!+\!N(e-e_j)+\!\sum_{k\ne i,j}N(e-e_i-e_j-e_k) 
=\!T+2N(e-e_i-e_j)\,. 
\end{gather*}}% 
Summing on $i$ in (b) and evaluating $\sum N(e-e_i)$ by (a), 
we obtain 
{\small 
$$(b')\quad 2N(e)+\sum_{i,j;j\ne i}N(e-e_i-e_j)=8T\,,$$} 
and summing on $i,j$ in $(c)$ and dividing out both sides by~$2$, 
we get 
{\small 
$$(c')\quad 5\sum_\ell N(e-e_\ell)+\frac12 \sum_{i,j,k\text{\,distinct}} 
\!\!\!N(e-e_i-e_j-e_k)=15T+\sum_{j\ne i}\,N(e-e_i-e_j)\,.$$} 
Evaluating $\sum_\ell N(e-e_\ell)$, adding $(b')\nd(c')$ 
%and dividing out both sides by~$6$, we obtain 
yields 
$$12N(e)+\frac 12\sum_{i,j,k\text{\,distinct}} 
N(e-e_i-e_j-e_k)=18T\le 108\,,$$ 
hence $N(e)\le\frac12\,(108-60)=\frac{48}{12}=4$. 
%{3T-10}2\le\frac{18-10}2=4$. 

\medskip 

In all cases (including in Proposition~\ref{propcrudebound}), 
if equality holds we necessarily have $T=n$, 
which is equivalent to $\forall\,i,\,N(e_i)=1$ and shows 
that $\Lb$ must be well-rounded. 

\smallskip 

If $n=4$ $\Lb$ is a centred cubic lattice by Lemma~\ref{lemindex2}. 

\smallskip 

If $n=5$, the proof above shows that all vectors 
$e_i$, %and 
$e-e_i-e_j$, $j\ne i$ have the same norm. A simple calculation 
will show that these condition determines uniquely the Gram matrix 
of the $e_i$ once their norm is given, i.e., that $\Lb$ is perfect. 
By the classification of $5$-dimensional perfect lattices 
(see \cite{M}, Section~6.4), since $\Lb$ is not a root lattice 
($\Lb$ is not integral when scaled to minimum~$2$), 
$\Lb$ is similar to $\A_5^3$, 
and we easily check that $\frac{N(e)}{N(e_i)}=\frac 52$. 

\smallskip 

The situation is somewhat similar if $n\ge 7$. The bound for $N(e)$ 
given in the proposition is attained only if $e_i\cdot e_j=\frac 12$ 
for all $i$ and $j\ne i$. Scaling the $e_i$ to norm~$2$ 
we recognize the Korkine and Zolotareff Gram matrix for $\A_n$ 
(with entries $a_{i,i}=2$ and $a_{i,j}=1$ off the diagonal). 
This shows that $\Lb'$ is then similar to $\A_n$, 
and we then have $\min\Lb=\min\Lb'$ 
(by results of Coxeter and Barnes; see \cite{M}, Section~{5.1}). 
Again $\Lb$ is perfect, and we easily check 
that the value of $\frac{N(e)}{N(e_i)}$ is the convenient one. 

\smallskip 

Finally if $n=6$ we content ourselves with an example. 
Taking $N(e_i)=3$ and $e_i\cdot e_j=1$ if $j\ne i$, 
then we see that $\Lb'$ is an integral lattice of minimum~$3$ 
for which $N(e)=12$. 
\newline{\small 
[By a joint theorem with Boris Venkov, the condition $s\ge 16$ 
cha\-racterizes $\Lb$ among integral lattices {\em of minimum~$3$} 
as a scaled copy of~$\D_6^+$.]} 
\end{proof}

%%Section 4 
\section{Dimensions up to seven} \label{secdim<=7} 

In this section we first give a short proof 
of Theorem~\ref{thn<=8} for dimensions~$n\le 6$, 
then prove some bounds for lattices of index $4$, 
and finally prove Theorem~\ref{thn<=8} for dimension~$7$.

%%Subsection 4.1 
\subsection{Dimensions up to 6} 
In this subsection we prove Theorem~\ref{thn<=8} for $n\le 6$. 
%The proof relies on the consideration of the maximal index 
%$\imath(\Lb)$ of $\Lb$. 
Recall (Watson; see \cite{M1}, Theorem~1.7) 
that we have \hbox{$\imath(\Lb)\le\g_n^{n/2}$}. 

If $\imath(\Lb)=1$, there is nothing to prove. 
Now one has $\g_n^{n/2}\le 2$ if $n\le 4$, and the value $2$ 
is attained by $\g(\Lb)$ only if $n=4$ and $\Lb$ is the centred 
cubic lattice (similar to the root lattice $\D_4$), 
which has a basis of minimal vectors. This shows that 
we have $H_b=M$ for all $n\le 4$. 

Next if $\imath(\Lb)=2$ (which needs $n\ge 4$), 
Theorem~\ref{thn<=8} results from Corollary~\ref{corindex2}. 
This applies to dimension $5$ since $\g_5^{5/2}=\sqrt 8<3$. 

For $n=6$, we have $\g_6^3=4.618\dots$, 
so that we need also consider indices $3$ and $4$. 
If $\imath=3$ we have $Q_b=1$ by Proposition~\ref{propWatsoncond}, 
and if $\imath=4$, we know by Remark~\ref{remdenom4} 
that $\Lb/\Lb'$ is $2$-elementary. Thus we may apply 
Proposition~\ref{prop2elem}: there is a unique $[6,2,4]$-code, 
it has weight distribution $(4^3)$, so that we again 
have $Q_b=1$. 
(The lifts of this code are similar to the root lattice $\D_6$; 
see see \cite{M1}, Table~11.1). 

This completes the proof of Theorem~\ref{thn<=8} 
in dimensions $n\le 6$.

%%Subsection 4.2 
\subsection{A bound for index 3} 
We consider a lattice $\Lb$, a frame $e_1,\dots,e_n$ of successive 
minima for $\Lb$ and the sublattice $\Lb'$ of $\Lb$ it generates. 
We shall prove the strict inequality $Q_b(\Lb)<\frac n4$ 
if $[\Lb:\Lb']=3$ and $7\le n\le 10$. 
However we consider for further use the slightly more general 
situation, of index $d\ge 3$, for which 

\smallskip\ctl{$\Lb=\la\Lb',e\ra$ where $e=\dfrac{e_1+\dots+e_n}d$} 
%\smallskip\noi 
%for some $d\ge 3$. 

%%Lemma 4.1 
\begin{lemma} \label{lemindex3a} 
%Let $\cB=(e_1,\dots,e_n)$ be a basis for $E$, 
%let $\Lb'$ be the lattice generated by $\cB$, 
%and let $e=\frac{e_1+\dots+e_n}d$. 
Recall that $T=\sum_{i=1}^n\,N(e_i)$. 
Then we have the identity 
$$\sum_{1\le i<j\le n}\,N(e-e_i-e_j)= 
(n-2)T+\frac{n^2-(4d+1)n+2d(d+2)}2\,N(e)\,.$$ 
\end{lemma} 

\begin{proof} 
Consider Watson's identities relative to $e$ and to the $e-e_i$: 
{\small\begin{gather*} 
\sum_{i=1}^n\,\big(N(e-e_i)-N(e_i)\big)=(n-2d)N(e)\\ 
\nd\forall\,i\,,\\ 
(d\!-\!1)\big(N(e)\!-\!N(e_i)\big)+ 
\sum_{j\ne i}\big(N(e\!-\!e_i\!-\!e_j)\!\!-\!\!N(e_j)\big)= 
(n\!-\!d\!-\!2)N(e\!-\!e_i)\,. 
\end{gather*}}% 
Summing on $i$ in the second identity and evaluating 
the right hand side using the first identity, we obtain 
{\small 
\begin{gather*} 
(d-1)n N(e)-(d-1)T+\sum_{i,j;j\ne i}\,N(e-e_i-e_j)-(n-1)T\\= 
(n-d-2)\sum_i\,N(e-e_i)\\= 
(n-d-2)(n-2d)N(e)+(n-d-2)T\,, 
\end{gather*}}% 
from which the required identity follows after dividing out by~$2$ 
the coefficients of $T$ and of~$N(e)$. 
\end{proof} 

%%Lemma 4.2 
\begin{lemma} \label{lemindex3b} 
%With the notation of Lemma~\ref{lemindex3a}, assume that $\cB$ 
%is a frame of successive minima for $\Lb=\la\Lb',e\ra$ 
%and that $n\le 3d+2$. 
There exists among the vectors $e-e_i-e_j$, $1\le i<j<n$, 
a vector $x$ such that 
$$N(x)\le\,\frac{2(n^2-3n+1)+\big(n^2-(4d+1)n+2d(d+2)\big)\,N(e)}% 
{(n-1)(n-2)}\,.$$ 
\end{lemma} 

\begin{proof} 
Since $(e_i)$ is a frame of successive minima for $\Lb$, 
we have $N(e_i)\le N(e_n)=1$ for all~$i$, hence $T\le n$, 
and $N(f)\ge N(e_n)$ for all $f\in\Lb\sm\Lb'$. 
In the identity of Lemma~\ref{lemindex3a} the left hand side is a sum 
of $\frac{n(n-1)}2$ terms from which we discard the $(n-1)$ 
terms $e-e_i-e_n$, obtaining 
{\small 
$$\sum_{1\le i<j<n}\,N(e-e_i-e_j)\le n(n-2)-(n-1)+ 
\big(n^2-(4d+1)n+2d(d+2)\big)\,N(e)\,.$$}% 
%(2n-4)T+(n^2-13n+30)N(e)-2\sum_{i=1}^{n-1}N(e-e_i-e_n)\,.$$} 
%\indent 
Dividing out the right hand side by $\frac{n(n-1)}2$ 
yields the inequality we want to prove for the smallest norm 
of a vector $e-e_i-e_j$, $i<j<n$. 
\end{proof} 

%%Lemma 4.3 
\begin{lemma} \label{lemindex3c} 
With the hypothesis of Lemma~\ref{lemindex3b}, 
assume moreover that we have $n\le 3d+1$. 
Then there exists among the vectors $e$ and 
$e-e_i-e_j$, $1\le i<j<n$, 
a vector $y$ such that 
$$N(y)\le\frac{n^2-3n+1}{(2d-1)n-(d^2+2d-1)}\,.$$ 
In particular if $d=3$ and $n\le 10$, or $d=4$, $m_2=0$ and $n\le 13$, 
then $\frac{H_b(\Lb)}{M(\Lb)}$ is strictly smaller than $\frac n4$. 
\end{lemma} 

\begin{proof} 
View $N(e)$ as a parameter $t\ge 1$, and denote by $\vf_{n,d}(t)$ 
the bound for $N(x)$ proved in Lemma~\ref{lemindex3b}. 
The coefficient $\a(n,d)$ of $t$ in the numerator of $\vf$, 
viewed as a function of $n$, attains its minimum on $\R$ 
for $n=2d+\frac12$, hence on $\Z$ for $n=2d$ and $n=2d+1$, 
equal to $-2d^2+2d<0$, and takes for $n=3d+1$ the value 
$-d^2+3d<0$. Thus $\vf_{n,d}(t)$ is a decreasing function 
of $t$ on $(1,+\infty)$ and attains its maximum at $t=1$, 
which is easily seen to be greater than~$1$. 
Since $t$ itself is an increasing function, 
$\min(t,\vf_{n,d}(t))$ is bounded from above by the value 
of $t$ for which $t=\vf_{n,d}(t)$, say, $\psi(n,d)$, 
which is the bound given in the Proposition. 
The comparison with $\frac n4$ is obvious. 
\end{proof} 

%%Proposition 4.4 
\begin{prop} \label{propindex3} 
With the notation of the lemmas above, assume that 
we have either $d=3$ and $7\le n\le 10$, or $d=4$ and $n\le 13$. 
Then $Q_b(\Lb)$ is strictly smaller than $\frac n4$. 
\end{prop} 

\begin{proof} 
The vector $y$ in Lemma~\ref{lemindex3c} is of the form 
$\frac{a_1 e_1+\dots+a_{n-1} e_{n-1}+e_n}d$, so that 
$(e_1,\dots,e_{n-1},y)$ is a basis for $\Lb$, 
and the bound of $N(y)$ of Lemma~\ref{lemindex3c} 
is thus a bound for $Q_b(\Lb)$. 
\end{proof} 

%%Remark 4.5 
\begin{remark} \label{remm_1=n-1} {\small\rm 
The methods of Proposition~\ref{propindex3}, the proof of which relies 
on the crude bounds of Proposition~\ref{propcrudebound} and 
Corollary~\ref{corcrudebound}, 
can be used more generally to handle the case when $d=4$ and $m_1=n-1$. 
One can prove this way the bound $Q_b<\frac 94$ when $n=9$ 
and $(m_1,m_2)=(8,1)$.} 
\end{remark} 

%%Subsection 4.3 
\subsection{Some more bounds for index 4} 
In this subsection we consider the case when $\Lb/\Lb'$ is cyclic 
of order~$4$. The notation $S_1,S_2,m_1,m_2$ ($m_1\ge 4$) 
is that of Definition~\ref{defm_i}. 

%%Proposition 4.6 
\begin{prop} \label{propindex4a} 
Assume that we have $7\le n\le 10$ and that $\Lb/\Lb'$ 
is cyclic of order $4$. Then: 
\begin{enumerate} 
\item 
If $m_1=4$, we have $\Q_b(\Lb)\le\dfrac{n-3}4<\dfrac n4$. 
\item 
If $m_1=5$, $\Q_b(\Lb)$ is bounded from above by $\frac 98$ if $n=7$, 
and by $\frac{(2n+5)^2}{320}<2$ if $n=8,9,10$. 
\end{enumerate}\end{prop} 

\begin{proof} 
We keep the notation $e_1,\dots,e_n$ for the successive minima, 
assuming that $N(e_1)\le\dots N(e_n)=1$, and $\Lb'=\la e_i\ra$. 
Set $e'=\dfrac{\sum_{i\in S_1}\,e_i}2$, so that 
$e=\frac{e'+\sum_{j\in S_2}\,e_j}2$, 
set $\cS=\Big\{\frac{\sum_{i\in S_1}\,e_i}2\Big\}$, 
and denote by $\a$ (resp. $\beta$) the largest subscript 
$i\in S_1$ (resp. $i\in S_2$). Thus $\a=n$ or $\beta=n$. 
By Lemma~\ref{lemindex2}, negating $e_i$ for some $i\in S_2$, 
we may assume that $e$ has the smallest norm among vectors of $\cS$. 

\smallskip 

Assume first that $n=\a$. Then replacing $e_n$ by $e$, we obtain 
a basis for $\Lb$ for which $Q_b(\Lb)\le\frac{N(e')+m_2}4$. 
By Proposition~\ref{propnorm1}, we have $N(e')\le 1$, 
hence $Q_b\le \frac{1+m_2}4=\frac{n-3}4$ if $m_1=4$, 
and $N(e')\le\frac 52$, hence 
$Q_b\le\frac{5/2+m_2}4=\frac{2n-5}4$ if $m_1=5$. 

\smallskip 

Assume now that $n=\beta$. We may no longer replace $e_n$ 
by $e$ since the numerator of $e$ now contains the term $2e_n$. 
We can instead replace $e_\a$ by any vector $e''\in\cS$ 
to be chosen later and $e_n$ by $e$, 
obtaining the upper bound $Q_b(\Lb)\le N(e'')\cdot N(e)$. 

If $m_1=4$, we choose $e''=e'$, and since $N(e')=N(e_\beta)$, 
we again have $Q_b\le\frac{n-3}4$. 

If $m_1=5$, taking $x=v$ with the notation of Lemma~\ref{lemtuvw}, 
(1b), we may achieve $N(e'')\le\frac{15-2t}{10}$, hence 
$N(e'')\cdot N(e)\le\vf(t):=\frac{(15-2t)(t+m_2)}{40}$. 
The maximum of $\vf$ on $\R$ is attained at $t=t_0:=\frac{15-2m_2}4$. 

If $n=7$, i.e., $m_2=2$, we have $t_0>\frac 52$, 
the bound for $t$ of Proposition~\ref{propnorm1}, 
and since $\vf(1)<1$, the maximum of $\vf$ on $[1,\frac 52]$ 
is $\vf(\frac 52)=\frac 98$. 

If $n=8,9,10$, i.e., $m_2=3,4,5$, we have $t_0\in(1,\frac 52)$, 
hence 

\smallskip\ctl{ 
$N(e'')N(e)\le\vf(t_0)=\frac{(2n+5)^2}{320}$ 
if $n=8,9$ or~$10$\,,} 

\smallskip\noi 
slightly larger than the bounds we obtained for $n=\a$. 
\end{proof} 

%%Remark 4.7 
\begin{remark} \label{remindex4a} 
{\small\rm 
The bounds of Proposition~\ref{propindex4a} 
are optimal if $m_1=4$, and if $n=7$ and $m_1=5$, and attained uniquely 
on well rounded lattices. The bounds for $n=8,9,10$ and $m_1=5$ are 
not optimal, and even the first bound $\frac{2n-5}4$, which applies 
to well-rounded lattices, could be improved, using vectors 
$e-e_i$ or $e-e_i-e_j$, $i,j\in S_1$.} 
\end{remark}

%%Subsection 4.4 
\subsection{Dimension 7} 
We now prove Theorem~\ref{thn<=8} for dimension~$7$, by inspection 
of all possible structures of $\Lb/\Lb'$ when $\Lb'$ is generated 
by a frame of successive minima $e_1,\dots,e_n$ for~$\Lb$, 
that we scale so as to have $N(e_n)=1$. 

\smallskip 

We know from \cite{M1} that $\Lb/\Lb'$ is of one of the types 
$(1)$, $(2)$, $(3)$, $(4)$, $(2^2)$, $(2^3)$, the latter case 
occurring only on the similarity class of~$\E_7$. 
Thus there is nothing to prove if $[\Lb:\Lb']=1\,\text{or}\,8$. 

The case of index~$2$ results from Corollary~\ref{corindex2}, 
and that of $2$-elemen\-tary quotients has been dealt with 
in Subsection~\ref{subsecden2n<=10}. 
{\small(There are two primitive codes of weight $w\ge 4$. 
There weight distributions are $4^2\cdot 6$ and $4\cdot 5^2$, 
so that we have $Q_b=1,\frac 54$, respectively.)} 

The case of index~$3$ results from Proposition~\ref{propindex3}, 
which implies $Q_b(\Lb)\le\frac{29}{21}=1.38...<\frac 74=1.75\,$. 

Consider finally the case when $\Lb/\Lb'$ is cyclic of order~$4$. 
We have $4\le m_1\le 6$ by Watson's 
identity~\ref{propWatsonindentity}, 
$Q_b=1$ if $m_1=4$ or~$6$ by Remark~\ref{remindex4a}, 
and $Q_b\le\frac 98<\frac 74$ if $m_1=5$ 
by Proposition~\ref{propindex4a}. 

\smallskip 

This completes the proof of Theorem~\ref{thn<=8} for all dimensions 
$n\le 7$. 

\medskip 

The bound above for index $3$ is still not optimal, 
and can be improved by making use also of vectors $e-e_i$ 
to $\frac{(n-2)(n^2-2n-9}{(n-1)(5n-18)}\,$ 
($\frac{65}{51}=1.27...$ for $n=7$, still not optimal). 

We summarize in the table below 
our knowledge on optimal bounds for dimension~$7$. 
The lower bounds for cyclic quotients of order $3$ and $4$ 
are attained on the Gram matrices $A_{7,3}$ and $A_{7,4}$ 
displayed after the table below. 

If we restrict ourselves to well-rounded lattices, 
we need not discard the subscript $n$ in the lemmas above. 
I could then show that $\frac{11}9$ 
\linebreak 
$(=1.22\dots)$ is optimal among well-rounded lattices. 
This is probably the general exact bound. 

%\hbox{}\vskip-.6cm 
\begin{table}[ht] 
\vbox{ 
\renewcommand{\arraystretch}{1.3} %Alternative possibility : 
%$\vphantom{2^{2^2}}\Lb/\Lb'$ &$1$&$2$&$3$&$4$&$2^2$&$2^3$\\ 
\noindent 
%\begin{minipage}{6cm} 
\begin{tabular}{|c|c|c|c|c|c|} 
\hline 
$1$ & $2$ & $3$ & $4$ & $2^2$ & $2^3$ \\ 
\hline 
$1$ & $\frac 74$ & $\frac{11}9\le Q_b<\frac{65}{61}$ & 
$\frac 98$ & $\frac 54$ & $1$ \\ 
\hline 
\end{tabular} 
%\end{minipage}%\hfill 
\medskip 
\caption{Optimal bounds in dimension 7} 
\label{tab:dim7} 
} 
\end{table} 

\hbox{}\vskip-.9cm{\small 
Here are Gram matrices for lattices which realize 
$\frac{H_b}M=\frac{11}9$ and $\frac{H_b}M=\frac 98$ 
for cyclic quotients of order $3$ and $4$, respectively: 

$$A_{7,3}=\left(\stx 22&-6&9&9&9&9&9\\-6&18&3&3&3&3&3\\ 
9&3&18&3&3&3&3\\9&3&3&18&3&3&3\\9&3&3&3&18&3&3\\9&3&3&3&3&18&3\\ 
9&3&3&3&3&3&18\estx\right)\,;\qquad 
A_{7,4}=\left(\stx 9&4&4&4&4&4&4\\4&8&2&2&2&0&0\\4&2&8&2&2&0&0\\ 
4&2&2&8&2&0&0\\4&2&2&2&8&0&0\\4&0&0&0&0&8&0\\4&0&0&0&0&0&8 
\estx\right)\,.$$ 
In both cases the minimum is read on the diagonal entries 
$a_{i,i}$, $i\ge 2$ and $\frac{H_b}M=\frac{a_{1,1}}{a_{2,2}}$.}

%%Section 5 
\section{Indices 4 and 5 and dimension 8} \label{secindex4-5} 
In this section we complete the proof of Theorem~\ref{thn<=8} 
by calculating sufficient bounds for $\frac{H_b}{M}$ 
in dimension~$8$. However, for further use, we sometimes 
consider dimensions~$n$ which may be greater than~$8$. 

\smallskip 

We keep the notation of the previous section: $\Lb$ 
is an $n$-dimensional lattice and $\Lb'$ denotes its sublattice 
generated by a frame {\small $\cB=(e_1\!,\dots,\!e_n)$} 
of successive minima for~$\Lb$, and we assume that $N(e_n)=1$. 
%We consider the different possible structures for $\Lb/\Lb'$, referring to Table~11.1 of \cite{M1}. There is nothing to prove if $[\Lb:\Lb']\ge 9$ since $\Lb$ is then similar to~$\E_8$, which has a basis of minimal vectors. The case of $2$-elementary quotients has been dealt with in Section~\ref{secdenom2} and that of index~$3$ in Proposition~\ref{lemindex3a}. %xXx lem or prop?Hence to complete the proof of Theorem~\ref{thn<=8} in dimension~$8$ it suffices to consider quotients $\Lb/\Lb'$ which are either cyclic of order $\imath=4$, $5$ or~$6$, or of type $(4\cdot 2)$. 
The notation $m_i$, $S_i$, $T_i$ that we use 
when dealing with cyclic quotients is that of Definition~\ref{defm_i}.

%%Subsection 5.1 
\subsection{Index 4} \label{subsecindex4} 

Though an {\em ad hoc}, relatively short proof could be 
given for~$n=8$, 
we prove below bounds which also apply to dimension~$9$. 

%%Proposition 5.1 
\begin{prop} \label{propindex4b} 
Assume that $\Lb/\Lb'$ is cyclic of order~$4$. 
Then if $n\le 9$, and if either $m_1\le 7$ or $m_1=n$, 
then $Q_b(\Lb)<\frac n4$. 
\end{prop} 

\begin{proof} 
The result has been proved previously 
if $m_1=n$ (Proposition~\ref{propindex3}), 
if $m_1\le 5$ (Proposition~\ref{propindex4a}), 
and if $n=7$. % (Subsection 4.4). 

%xXx 

%so that it suffices to consider the $4$~values of $(m_1,m_2)$ 
%such that $n\ge 8$ and $6\le m_1<n\le 10$. 
%We successively consider the cases when $m_1=6,7,8,9$. 
The proof for $m_1=6$ is an extension 
of that of Proposition~\ref{propindex4a} whereas we need sharper 
inequalities for $m_1=7$. 

In all cases we assume that the norm of $e$ is minimal among those 
of the vectors $\frac{e'+\sum_{i\in S_2}\,\pm e_i}2$. 

\smallskip 

$\underline{m_1=6}$ (thus, $m_2\in\{2,3\}$). 
With the notation of Lemma~\ref{lemtuvw}, 
we choose $x=v$ if $u\le 2$ and $x=w$ if $u\ge 2$, 
bounding this way $N(x)N(e)$ by the functions $\vf_1$ and $\vf_2$ 
below, to to be considered on the interval $[1,4]$ 
by Proposition~\ref{propnorm1}: 
$$u\le 2\,:\ \vf_1(t)=\frac{(10-t)(t+m_2)}{20}\,;\ 
u\ge 2\,:\ \vf_2(t)=\frac{(15-t)(t+m_2)}{40}\,.$$ 
The maximum on $\R$ of $\vf_1$ is attained at $t_1=\frac{10-m_2}2$ 
and that of $\vf_2$ at $t_2=\frac{15-m_2}2$. 
Since $m_2=2$ or~$3$, we have $t_1\in[1,4]$ and $t_2>4$. 
Hence $\frac{H_b(\Lb)}{M(\Lb)}$ is bounded from above 
by $\vf_1(t_1)$ if $u\le 2$ and $\max(\vf_2(1),\vf_2(4))=\vf_2(4)$ 
if $u\ge 2$, that is, in terms of $n=6+m_2$, 

\medskip\ctl{\small 
$\vf_1(t_1)=\frac{(n+4)^2}{80}=\frac n4-\frac{(10-n)(n-2)+4}{80} 
\nd \vf_2(4)=\frac{11(n-2)}{40}=\frac n4-\frac{22-n}{40}\,,$} 

\smallskip\noi 
which proves the result (even up to $n=10$). 

\medskip 

To handle lattices with $m_1\ge 7$ we return to Watson's identity 
for denominator~$4$, namely 
$$\sum_{i\in S_1}\,N(e-e_i)+2\sum_{i\in S_2}\,N(e-e_i)= 
T_1+2T_2+(m_1+2m_2-8)\,N(e)\,,$$ 
which implies, since $N(e-e_i)\ge N(e)$ for all $i\in S_2$, 
$$\sum_{i\in S_1}\,N(e-e_i)\le m_1+2m_2+(m_1-8)\,N(e)\,.$$ 

\smallskip 

$\underline{m_1=7,m_2=1}$. %(thus, $m_2\in\{1,2,3\}$). 
We have $N(e)\le\frac{t+1}4$, hence $t\ge 3$, 
and by the crude estimate, $t\le 7$, hence $u\le 1+\frac{37}7\le 4$. 
Using $\tilde e=\frac{(e'-2 e_i)+e_\beta}4$ of norm 
$N(\tilde e)\le \frac{2u-t+3}4$ %xXx %error : of norm $\le\frac{2u-t+3}4$ ?? 
instead of $e$ and taking $x=w$, of norm $N(x)\le\frac{21-u-t}{10}$ 
(Lemma~\ref{lemtuvw}) 
we reduce ourselves to bound the function 
\linebreak 
$\vf(t)=\frac{(21-u-t)(2u-t+3)}{40}$ 
in the domain defined by the inequalities 
\linebreak 
$1\le u\le 1+\frac{3t}7\le 4$ and $3\le t\le 7$. 
As a function of $u$, $\vf'$ is zero for $u=\frac{39-t}4\ge 8>4$, 
so that we have
$\vf(u,t)\le\vf(1+\frac{3t}7,t)=\frac{(35-7)(14-t)}{4\cdot 7^2}$, 
a decreasing function on $[3,7]$. Its maximum on $[3,7]$ 
is attained for $t=3$ and is equal to 
$\frac{32\cdot 11}{4\cdot 7^2}=\frac{88}{49}<2$. 

\smallskip 

$\underline{m_1=7,m_2=2}$. This time we find 
$\vf(t)=\frac{(21-u-t)(2u-t+4)}{40}$, 
to be considered in the domain 
$1\le u\le 1+\frac{3t}7$ and $2\le t\le 7$. As a function of $u$, 
its maximum is 
$M_1=\vf(1+\frac{3t}7,t)=\frac{(42-t)(14-t)}{4\cdot 7^2}$. 
For $t\ge 3$ we have $M_1\le\frac{429}{196}=2.18...<\frac 94$. 
For $t\in[2,3]$, we use $N(x)N(e)\le\frac{(21-u-t)(t+2)}{40}$, 
which as a function of $t$ is maximum at $t=3$ and then equal 
to $\frac{18-u}8<\frac{17}8<\frac 94$. 
\end{proof} 

This completes the proof of Theorem~\ref{thn<=8} for index~$4$. 

%$\underline{m_1=8,m_2=1}$. By Lemma~\ref{lemtuvw} we have 
%$u\le 1+\frac t2$, $v\le\frac{8+5u-t}7\le\frac{12+t}7$, then 
%$w\le\frac{8+4v-2u}6\le\frac{44+u-2t}{21}\le\frac{30-t}{14}$. 
%\ctl{\sc Does not suffice for n=9 and $m_1=8$.} 

%%Subsection 5.2 
\subsection{Index 5} \label{subsecindex5} 
We use as for denominator~$4$ the notation 
$m_1,m_2$, $S_1,S_2$. The cosets modulo $\Lb'$ 
are those of $0$, $\pm e$ and $\pm e'$ where 
$$e=\frac{\sum_{i\in S_1}\,e_i+2\sum_{i\in S_2}\,e_i}5 
\ \nd\ e'=\frac{2\sum_{i\in S_1}\,e_i-\sum_{i\in S_2}\,e_i}5\,.$$ 
We have $e'\equiv 2e\mod\Lb'$ 
so that exchanging $e$ and $e'$ and negating the $e_i$ with $i\in S_2$ 
if need be, we may assume that we have $m_1\ge m_2$. 
Then by Proposition~\ref{propWatsonindentity}, 
we have $(m_1,m_2)=(4,4)$, $(5,3)$ or $(6,2)$ if $n=8$ 
(and $(5,4)$, $(6,3)$, $(7,2)$ or $(8,1)$ if $n=9$), 
and by Proposition~\ref{propWatsoncond}, $\Lb$~has a basis of minimal 
vectors if $(m_1,m_2)=(6,2)$ (or~$(8,1)$). 

If $m_1=m_2=4$, by an identity of Zahareva 
(\cite{M1}, Section~9), we have $N(e-e_i)=N(e_i)$ for $i=5,6,7,8$ 
and $N(e'-e_i)=N(e_i)$ for $i=1,2,3,4$. Since $(e_i)$ is a frame 
of successive minima, we have 
$N(e_1)\le N(e_8)\le N(e'-e_1)=N(e_1)$, 
which shows that $e_i$, $e-e_j,j\ge 4$ and $e'-e_k,k\le 4$ 
are minimal vectors and that $(e_1,\dots,e_7,e-e_1)$ is a basis 
of minimal vectors for $\Lb$. 

\smallskip 

To handle the case when $(m_1,m_2)=(5,3)$, we shall use 
the crude bound of Proposition~\ref{propcrudebound}, which reads 

\medskip\ctl{ 
$N(e)\le\frac{m_1(m_1+1)/2+2m_2(m_2+1)}{25}$\,,} 

\smallskip\noi 
%\smallskip\noi 
together with Watson's identity 

\smallskip\ctl{ 
$\sum_{i\in S_1}\,(e-e_i)+2\sum_{i\in S_2}\,(e-e_i)= 
T_1+2T_2+(m_1+2m_2-10)N(e)\,,$} 

\medskip\noi 
%\smallskip\noi 
considering separately the cases $n\in S_1$ and $n\in S_2$. 

\smallskip 

If $n\in S_1$, using the obvious lower bound $N(e-e_i)\ge N(e_i)$ 
when $i\in S_2$ or $i=n$, we see that there exists among 
the vectors $e-e_i$, $i\in S_1\sm\{n\}$ an $x$ of norm 
$N(x)\le 1+\frac{m_1+2m_2-10}{m_1-1}\,N(e)$. 
Then $(e_1,\dots,e_7,x)$ is a basis for~$\Lb$. 
When $(m_1,m_2)=(5,3)$, the bounds above are 

\smallskip\ctl{ 
$N(e)\le\frac{69}{25}$ 
and $N(x)\le1+\frac{N(e)}4\le\frac{169}{100}=1.69<2$\,.} 

\smallskip 

If $n\in S_2$ (i.e., if $e_n=e_8$) we first observe that we have 
$e'=\frac{-e+\sum_{i\in S_1}e_i}2$, so that some vector 
$x=e'-e_{i_1}-\dots e_{i_k}$, $i_1,\dots,i_k\in S_1$, 
has a norm $N(x)\le\frac{N(e)+m_1}4$ (Lemma~\ref{lemindex2}). 
Again $(e_1,\dots,e_7,x)$ is a basis for $\Lb$, 
and when $(m_1,m_2)=(5,3)$, we have 

\smallskip\ctl{ 
$N(x)\le\frac{N(e)+5}4\le\frac{69/25+5}4=\frac{194}{100}<2$\,.}

%%Subsection 5.3 
\subsection{Dimension 8} \label{subsdim8} 

We now prove Theorem~\ref{thn<=8} for $8$-dimensional lattices, 
namely that in dimension~$8$, $Q_b=\frac{H_b}M$ is bounded 
from above by~$2$, with equality only on centred cubic lattices. 
%which completes the proof of the theorem. 
We consider the various possible structures of $\Lb/\Lb'$, 
and recall from \cite{M1} that if $[\Lb:\Lb']>8$ 
then $\Lb$ is similar to $\E_8$, which has a basis of minimal vectors, 
so that it suffices to consider indices $[\Lb:\Lb']\le 8$, 
excluding cyclic quotients of order $7$ or $8$ which do not 
exist in dimension~$8$. 

\smallskip 

The case of $2$-elementary quotients has been dealt with 
in Section~\ref{secdenom2}, so that it suffices to consider 
quotients $\Lb/\Lb'$ which are either cyclic of order $3$ to $6$ 
or of type $4\cdot 2$ and to show that we then have the strict 
inequality $Q_b(\Lb)<2$. 
We now consider successively the five possible cases for the maximal 
index of $\Lb$. 

\smallskip 

\noi\bl $\imath=3$. 
This is Proposition~\ref{propindex3}. 

\smallskip 
\noi\bl $\imath=4$. 
This results from Proposition~\ref{propindex4b}. 

\smallskip 
\noi\bl $\imath=5$. 
This results from Subsection~\ref{subsecindex5}. 

\smallskip 
\noi\bl $\imath=6$. 
With the notation $S_i,m_i$ for $i=1,2,3$, we have 
$\Lb=\la\Lb',e\ra$ where $e=\frac{\sum_i\,i\sum_{j\in S_i}\,e_j}6$. 
Besides $e$ we also consider 

\smallskip\ctl{ 
$f=\frac{\sum _{i\in S_1}\,e_i-\sum _{j\in S_2}\,e_j}3\ \nd\ 
g=\frac{\sum _{i\in S_1}\,e_i+\sum _{k\in S_3}\,e_k}2$\,,} 

\smallskip\noi 
in order to apply previous results for denominators $2$ and~$3$. 
Set $f_i=f-e_i$ if $i\in S_1$ and $f+e_i$ if $i\in S_2$. 

There are six $\Z/6\Z$-codes listed in \cite{M1}, 
among which five define lattices having a basis of minimal vectors. 
(This can be easily checked using Section~9 of \cite{M1}, 
as we did above for $(m_1,m_2)=(4,4)$ with denominator~$5$.) 
The remaining code has $(m_1,m_2,m_3)=(3,3,2)$. 
Since $m_1+m_3=6$, Watson's identity shows that the vectors 
$e_i$ and $f_i$ for $i\in S_1\cup S_2$ have equal norm. 
Consider two subscripts $i,j\in S_2$, and in the basis 
$(e_i)$ for $\Lb'$, replace $e_i$ by $f_j$. 
The we obtain a new frame of successive minima, which spans 
a lattice $L$ such that $\Lb=\la L,g\ra$. 
We are thus reduced to index~$2$, 
and this proves that we have 
$Q_b(\Lb)\le\frac{m_1+m_3}4=\frac 54$ in this case. 

\smallskip 
\noi\bl $\imath=8$. Here there are three codes over $\Z/4\Z$, 
and in each case we have $\Lb=\la\Lb',e,f\ra$ 
for vectors $e$ and $f$ of denominators $4$ and $2$, respectively. 
In all cases $\Lb$ has a basis of minimal vectors: 
in the first case because $e$ and $f$ are minimal, 
and in the remaining two cases because these are known lattices, 
namely a lattice on a Voronoi path \hbox{$\E_8$---$\E_8$} with $s=75$ 
discovered by Watson, and~$\E_8$. 

\smallskip 

This completes the proof of Theorem~\ref{thn<=8}. \qed 

%\vfil\eject\hbox{} 

%%Section 6 
\section{Beyond dimension 8} \label{secdim9} 

In this section we collect various results and remarks 
concerning dimensions $n>8$. In the first subsection 
we prove Theorem~\ref{thn=9}, then extend it to some cases 
concerning indices between $5$ and~$8$. 
We then consider in the second subsection some extensions 
Theorem~\ref{thn=9} for well-rounded lattices. 
Finally in a short last subsection we make 
a few remarks on larger dimensions.

%%Subsection 6.1 
\subsection{Proof of Theorem~\ref{thn=9}} 
\label{subsecdim9a} 
We now proceed to the proof of 
\linebreak 
Theorem~\ref{thn=9} by looking 
successively at the various structures of $\Lb/\Lb'$ 
listed in its statement. As usual we restrict ourselves 
to primitive codes, since otherwise the result follows 
from the bounds we proved for dimensions $n\le 8$. 

\begin{proof} %{\em First part: unconditional bounds.} 

\bl\underbar{$2$-elementary quotients}. These have been dealt with 
in Subsection~\ref{subsecden2n<=10}. 

\smallskip 

\bl \underbar{$3$-elementary quotients}. 
The case when $[\Lb:\Lb']=3$ is Proposition~\ref{propindex3}. 
Otherwise we have $[\Lb:\Lb']=9$ and there are three admissible 
ternary codes, listed in Table~6 of \cite{K-M-S}, 
all of which have a basis $(w_1,w_2)$ with $\wt(w_1)=6$ 
and $\wt(w_2)=6,6$ and~$7$, respectively. 
Since $\wt(w_1)=6$ Watson's identity shows that the lattice 
generated by $\Lb'$ and a lift of $w_1$ has a basis of successive 
minima, so that we are reduced to the case of index $3$ 
and dimension $\wt(w_2)$, for which we know a bound for $Q_b(\Lb)$ 
which is much smaller than~$\frac 94$. 

\smallskip 

\bl\underbar{Index $4$}. 
We need only consider cyclic quotients, 
classified by pairs $(m_1,m_2)$ with $m_1\ge 4$ and $m_1+m_2=9$. 
The bound $Q_b<\frac 94$ has been proved in 
Proposition~\ref{propindex4b} if $m_1=9$ or $m_1\le 7$, 
so that we are left with the case when $(m_1,m_2)=(8,1)$, for which 
the methods of Subsection~\ref{subsecindex4} do not suffice. 
This case can be solved by bounding from above the smallest norm 
of a vector $e-e_i-e_j$, with the same line of proof than that 
of Lemmas \ref{lemindex3a} to~\ref{lemindex3c}. 
The details are left to the reader. 
Note however that the results of Subsection~\ref{subsecindex4} 
suffice for well-rounded lattices. 

\bl \underbar{Index $>9$}. The {\em PARI-GP} companion file 
{\em Gramindex.gp} to \cite{K-M-S} shows that there exists 
for every lattice a Gram matrix having diagonal entries equal 
to its minimum, except for the matrix $a9f62$, which acquires 
such a diagonal after an $L L L$-reduction. Hence all lattices 
$\Lb$ with $[\Lb:\Lb']\ge 10$, except possibly those having 
a $2$-elementary quotient of order~$16$, indeed have a basis 
of minimal vectors, hence satisfy \hbox{$Q_b(\Lb)=1$}. 
(For codes over $\F_2$ one has $Q_b=1$ or $\frac 54$.) 

\bl\underbar{Index $9$}. We need only consider cyclic quotients. 
Six codes are displayed in Table~2 of \cite{K-M-S}. 
For the first four, with $s=84,50,136,53$, respectively, 
the file {\em Gramindex.gp} shows the existence of bases 
of minimal vectors. For the remaining two, 
and further in the sequel, departing from our previous convention, 
we order the $e_i$ choosing successively 
vectors from $S_1$, then $S_2,\,\dots$, and write $e$ as successive 
sums having denominator~$3$ namely 
{\small 
$$e=\frac{e'+e_5+e_6+e_7+e_8+e_9}3\text{ with } 
e'=\frac{e_1+e_2+e_3+2e_4+e_8+e_9}3$$} 
and 
{\small 
$$e=\frac{e'+e_4+e_5+e_6+e_7+e_8+e_9}3\text{ with } 
e'=\frac{e_1+e_2+2e_3-e_4+e_8+e_9}3\,.$$} 
Using Watson's identity we see that in both cases, 
the successive minima on the support of $e'$ are equal, 
and that the same property holds for $e$ in the first case 
(and then $Q_b=1$) whereas we may apply 
Proposition~\ref{propindex3} for dimension~$7$ 
in the second case. (There are then $15$~minimal vectors, 
which all lie in $\Lb'$ or $\pm e'+\Lb'$.) 
\end{proof} 

%xXx 
%%Subsection 6.2 
\subsection{More on index 9} 
\label{subsecdim9b} 

We now refer to Tables 2 and~6 of \cite{K-M-S}, and use the notation 
$C_{d,i}$, $d=5$, $6$, $7$, $8$ or $4\cdot 2$, to denote 
the i-{\thh} class of lattices with $\Lb/\Lb'$ of type $(d)$ 
in the corresponding table. Here $i$ runs from $1$ to $i(d)$, where 
$i(5)=4$, $i(6)=20$, $i(7)=8$, $i(8)=19$, and $i(4\cdot 2)=26$. 

In this table (as in \cite{M1}) $s$ (resp. $s'$) is the number 
of pairs of necessary minimal vectors for $\Lb$ (resp. $\Lb'$). 
Thanks to a deformation argument the tables could be constructed 
using only well-rounded lattices; 
in our context $s$ (resp. $s'$) is the minimal number of pairs 
of representatives of the successive minima for $\Lb$ (resp. $\Lb'$). 

%%Proposition 6.1 
\begin{prop} \label{propindex5-8} 
Let $\Lb$ be a lattice belonging to a class $C_{d,i}$, 
$d=5$, $6$, $7$, $8$ or $4\cdot 2$. Then if $s>9$, 
$Q_b(\Lb)$ is strictly smaller than $\frac 94$. 
%\newline
{\small\rm 
[The number of classes satisfying these conditions are 
$1$, $11$, $5$, $16$, and $23$, respectively.]} 
\end{prop} 

\begin{proof} 
We first observe that $s>9$ implies $s>s'$. 
This shows that in all cases, the lattice $L$ generated 
by the successive minima of $\Lb$ satisfies $[\Lb:L]<[\Lb:\Lb']$. 
By inspection (or by \cite{M-S}) we see that $L$ actually has 
a basis made of successive minima. 

\smallskip 

As above for index~$9$ we consider the diagonal entries $a_{i,i}$ 
in the file {\em Gramindex.gp}. In all cases we have 
$a_{1,1}\ge a_{2,2}=\dots=a_{9,9}=\min\Lb$. 
If $a_{1,1}=a_{2,2}$ then $\Lb$ has obviously a basis of minimal vectors, 
which implies $Q_b(\Lb)=1$. This applies to denominators $5$ and $7$. 

Otherwise we list its minimal vectors and consider the leading 
components. If some leading component is equal to~$1$, 
we again have a basis of minimal vectors. This holds more 
generally if the leading components are coprime, which solves 
one more case with $d=6$ (and an $L L L$-reduction then produces 
explicitly a basis of minimal vectors). 
If the $\gcd$ of the leading components is $>1$, 
then $[\Lb:L]$ takes one of the values $2$, $3$ or $4$, 
and we need a closer look at minimal vectors. 

\smallskip 

\bl $d=6$. There remains six classes to consider. For two of them 
we have $m_1+m_3=4$, so that $\Lb$ contains to index~$3$ 
a lattice having a basis of successive minima, to which we may apply 
Proposition~\ref{propindex3}, and in the remaining four cases, 
we have $m_1+m_2=6$, so that $\Lb$ contains to index~$2$ a lattice 
having a basis of successive minima, to which we may apply 
Lemma~\ref{lemindex2}, after having checked that we may write 
$\Lb=\la L,\frac{e_{i_1}+\dots+e_{i_k}}2\ra$ for some $k\le 8$, 
which then ensures the upper bound $Q_b(\Lb)\le\frac 82=2$. 
\newline{\small 
The worst case is afforded by the code $(2,4,3)$, for which 
we write 

\smallskip\ctl{$e= 
\dfrac{(e_1+e_2+2e_3-e_4-e_5-e_6)/3+e_4+e_5+e_6+e_7+e_8+e_9}2\,,$} 

\smallskip\noi 
obtaining the (indeed strict) bound $Q_b(\Lb)\le\frac 74$.} 

\smallskip 

\bl $d=8$. There remains two classes to consider, with corresponding 
codes of type $(2,4,2,1)$ (matrix $a9j8$) 
and $(3,1,3,2)$ (matrix $a9s8$). 

In the first case we set $e'=\frac{e_1+e_2+e_7+e_8}2$ 
and $L=\la\Lb',e'\ra$ and write 
\linebreak 
$e=\frac{e'+e_3+\dots+e_8+2e_9}4$, so that we are reduced 
to the case of a cyclic quotient of order~$4$ in dimension~$8$. 

In the second case we set $e'=\frac{e_1+e_2+e_3-e_5-e_6-e_7+2e_4}4$ 
and write \linebreak $e=\frac{e'+e_5+e_6+e_7+e_8+e_9}2$ 
so that we are reduced to index~$2$ in dimension~$6$. 

\smallskip 

\bl $d=4\cdot 2$. There remains five classes to consider, 
for which Table~7 of \cite{K-M-S} displays a representation 
$\Lb=\la\Lb',e,f\ra$ with $4e$ and $2f$ in $\Lb'$. 
In all cases 
(matrices $a9g42$, $a9j42$, $a9r42$, $a9s42$ and $a9t42$) 
the support of $e$ is of length $7$ or~$8$ 
and that of $f$ of length~$4$, 
so that writing $\Lb=\la L,e\ra$, we are reduced to the case 
of index~$4$ in dimension $7$ or~$8$. 

\smallskip 

This completes the proof of the proposition. 
\end{proof} 

%%Remark 6.2 
\begin{remark} \label{remindex42} {\small\rm 
The bound $Q<\frac 94$ also holds for all quotients $\Lb/\Lb'$ 
of type $4\cdot 2$. Indeed in the three cases where $s=9$ 
in Table~7 of \cite{K-M-S}, $f$ has a support of length~$5$, 
which implies $Q_b(\Lb)\le\frac 54\,B$ where $B$ is the bound 
previously obtained for cyclic quotients of order~$4$ 
with $(m_1,m_2)=(5,2)$, $(6,2)$ and $(7,1)$, 
namely $\frac 98$, $\frac{66}{40}$ and $\frac{88}{49}$, 
that is $Q\le\frac{45}{32}<1.41$, $Q\le\frac{33}{16}<2.1$ 
and $Q\le\frac{110}{49}=2.24...$, respectively. 
The exact bounds are probably much smaller.} 
\end{remark} 

Taking into account 
Proposition~\ref{propindex5-8} and Remark~\ref{remindex42}, 
we are left with $3+9+3+3=18$ unsolved cases, corresponding 
to cyclic quotients $\Lb/\Lb'$ of orders $5,6,7,8$, that we list 
below: 

\smallskip{\small 

\noi $d=5$: $m_1=5,6,7$. 

\noi $d=6$: $m_3=0$, $m_1=5$ ; $m_3=1$, $m_1=4,5,6$ ; % 
$m_3=2$, $m_1=3$, $4$, $5$, $6$, $7$. 

\noi $d=7$: $(m_1,m_2,m_3)=(4,3,2)$, $(5,2,2)$, $(4,2,3)$. 

\noi $d=8$: $(m_1,m_2,m_3,m_4)=(3,4,2,0)$, $(3,3,2,1)$, $(3,2,2,2)$.} 

\smallskip\noi 
To deal with these eighteen remaining cases would make this paper 
unreasonably long. We consequently end here general proofs, 
though some more cases could have been solved along the line 
of Remark~\ref{remindex42}. 

\medskip 

%%Subsection 6.3 
\subsection{Well-rounded, 9-dimensional lattices} 
\label{subsecdim9c} 

In this subsection we consider well-rounded lattices, 
with $[\Lb:\Lb']=5$ or~$7$. 

%%Proposition 6.3 
\begin{prop} \label{inex5-7} 
Let $\Lb$ be a well-rounded lattice of dimension~$9$ and maximal index 
$5$ or $7$. Then $Q_b(\Lb)$ is strictly smaller than $\frac 94$. 
\end{prop} 

\begin{proof} 
We shall write down a detailed proof for $d=5$, 
and leave to the reader the case of index~$7$, 
for which it suffices to mimic the previous case. 
The method consists in applying Watson's identity 
and using the crude estimate~\ref{propcrudebound} 
to bound $N(e)$. 

In all cases the ordering of the vectors $e_i$ does not matter, 
since they all have the same norm, that we fix equal to~$1$. 
We order them as we did above for index~$9$, 

\medskip 

Thus let $[\Lb:\Lb']=5$, write as usual $n=m_1+m_2$, 
and assume that $m_1\ge m_2$. 
Fix a subscript $i$ in $\{1,\dots,n\}$, then a subscript $j\ne i$ 
in $S_1$. The vectors $e_k,k\ne j$ and $e-e_i$ then constitute 
a basis for $\Lb$, so that we have $Q_b(\Lb)\le N(e-e_i)$. 

Watson's identity, which reads 

\smallskip\ctl{$\sum_{k\in S_1}N(e_k)+2\sum_{k\in S_2}N(e_k)= 
(m_1+2m_2)+(m_1+2m_2-10)N(e)$\,,} 

\smallskip\noi 
shows that there exists $i$ such that 

\smallskip\ctl{$N(e-e_i)\le 1+\frac{m_1+2m_2-10}{m_1+2m_2}$\,N(e)\,,} 

\smallskip\noi 
whereas Proposition~\ref{propcrudebound} gives the bound 

\smallskip\ctl{$N(e)\le\frac{2n(n+1)-m_1(4n+3-m_1)/2}{25} $\,.} 
 
\smallskip\noi 
For given $n$, both $N(e)$ and its coefficient are decreasing 
functions of $m_1$, so that $N(e-e_i)$ is bounded from above 
by its value at $\lf\frac{n+1}2\rf$. 

For $n=9$ we obtain the bound 
$Q_b(\Lb)\le 1+\frac 3{13}\frac{95}{25}=\frac{122}{65}<\frac 94$. 

\smallskip 

The same argument applies to dimension~$7$, and the large 
deno\-minator $(49$ instead of $25$) in $N(e)$ 
yields in all cases an upper bound far below~$\frac 94$. 
\end{proof} 

Most of the proofs we gave all along this paper for dimensions 
$7,8,9$ could have been made simpler if we had restricted ourselves 
to well-rounded lattices, and we can even very often easily check that 
the bounds we obtained are not optimal every time we had to take 
into account the place of $e_n$ with respect to the subsets $S+i$ 
relative to various cyclic components. This supports the following 
conjecture: 

%%Conjecture 6.3 
\begin{conj} \label{conjwellrounded} 
In all dimensions, the maxima of $Q_b$ and $Q_g$ are attained 
on well-rounded lattices 
\end{conj} 

\subsection{Beyond dimension 9} 
\label{subseccompldim9} 

We just want to consider $2$-elementary quotients $\Lb/\Lb'$ 
up to dimension~$12$. The exact bound for $Q_b$ on $2$-elementary 
quotients has been shown to be equal to 
$\frac{6\cdot 7}{16}=\frac{21}8=2.625$ in dimension~$10$. 
This is a simple matter of classifying binary codes of weight~$w>4$. 
This classification is easily extended in dimensions $11$ and~$12$. 
It turns out that the highest values for $Q_b(\Lb)$ on the set 
of lattices with $2$-elementary quotients $\Lb/\Lb'$ 
are obtained by lifting unique even codes $C_{11}$ and $C_{12}$ 
of weight~$6$. We first define $C_{12}$ by the generator matrix 

\ctl{ 
$G_{12}=\left(\stx 
1&1&1&1&1&1&0&0&0&0&0&0\\ 
0&0&0&1&1&1&1&1&1&0&0&0\\ 
1&1&0&1&0&0&1&0&0&1&1&0\\ 
1&0&0&1&1&0&0&1&0&1&0&1\\ 
\estx\right)\,,$} 

\smallskip\noi 
then $C_{11}$ by the generator matrix $G_{11}$ obtained from $G_{12}$ 
by deleting the last column and the last row. 
% 1^30^31^30^3 ; 0^2101^30^21^20; 1^20^21^201^40; 0^21^20^31^40; 
% 01^20^2010101; 10^41^201^201  ; 010^2101^20^21^2; 
% 01^40101^201 ; 101^201^30^21^2; 0101010^2101^2 ; 101010^3101^2 
The weight distribution of $C_{11}$ is $6^6\cdot 8$ 
and that of $C_{12}$ is $6^{12}\cdot 8^3$, 
which gives $Q_b$ the lower bounds $\frac{27}8=3.375$ 
and $\frac{81}{16}=5.0625$, respectively, 
reasonably close to van der Waerden's bounds 
($4.768...$ and $5.960...$, respectively). 

\smallskip\noi 
{\em I conjecture that $2$-elementary quotients 
still produce the largest possible values for $\Q_b$ 
in dimensions $10$, $11$ and $12$.}

\vfill\eject\hbox{} 

%bib 

\end{document}